\documentclass[reqno]{amsart}
\usepackage{amssymb,amsmath,hyperref}
\usepackage{amsrefs, float}
\usepackage[foot]{amsaddr}
\usepackage{bbold,stackrel, multicol}
\usepackage{mathrsfs}

\newcommand{\Z}{{\textsf{\textup{Z}}}}

\newtheorem{thm}{Theorem}

\newtheorem{defi}[thm]{Definition}
\newtheorem{rem}[thm]{Remark}
\newtheorem{nota}[thm]{Notation}

\newtheorem{princ}[thm]{Principle}

\newtheorem{ack}[thm]{Acknowledgement}

\newtheorem*{tempo*}{Template}

\newcommand\be{\begin{equation}}
\newcommand\ee{\end{equation}} 

\usepackage{amsmath,amsfonts} 

\usepackage[applemac]{inputenc}

\def\bdefi{\begin{defi}}
\def\edefi{\end{defi}}
\def\bnota{\begin{nota}\rm}
\def\enota{\end{nota}}
\def\FIVE{\Pi_{1}^{1}\text{-\textup{\textsf{CA}}}_{0}}
\def\SIX{\Pi_{2}^{1}\text{-\textsf{\textup{CA}}}_{0}}

\def\ATR{\textup{\textsf{ATR}}}

\def\ZFC{\textup{\textsf{ZFC}}}
\def\ZF{\textup{\textsf{ZF}}}

\def\RCA{\textup{\textsf{RCA}}}
\def\({\textup{(}}
\def\){\textup{)}}

\def\RCAo{\textup{\textsf{RCA}}_{0}^{\omega}}
\def\ACAo{\textup{\textsf{ACA}}_{0}^{\omega}}

\def\WKL{\textup{\textsf{WKL}}}

\def\bye{\end{document}}
\def\N{{\mathbb  N}}

\def\R{{\mathbb  R}}

\def\SS{\textup{\textsf{S}}}

\def\di{\rightarrow}

\def\asa{\leftrightarrow}

\def\ACA{\textup{\textsf{ACA}}}

\def\QFAC{\textup{\textsf{QF-AC}}}
\def\AC{\textup{\textsf{AC}}}

\def\PHP{\textup{\textsf{PHP}}}

\def\NCC{\textup{\textsf{NCC}}}

\def\cocode{\textup{\textsf{cocode}}}

\def\NIN{\textup{\textsf{NIN}}}
\def\NCC{\textup{\textsf{NCC}}}
\def\NBI{\textup{\textsf{NBI}}}

\def\IND{\textup{\textsf{IND}}}

\def\eps{\varepsilon}

\usepackage{graphicx}
\usepackage{tikz}
\usetikzlibrary{matrix}
\usepackage{comment,tikz-cd}

\numberwithin{equation}{section}
\numberwithin{thm}{section}

\begin{document}
\title{A note on continuous functions on metric spaces}
\author{Sam Sanders}
\address{Department of Philosophy II, RUB Bochum, Germany}
\email{sasander@me.com}
\begin{abstract}
Continuous functions on the unit interval are relatively \emph{tame} from the logical and computational point of view. 
A similar behaviour is exhibited by continuous functions on compact metric spaces \emph{equipped with a countable dense subset}.  
It is then a natural question what happens if we omit the latter `extra data', i.e.\ work with `unrepresented' compact metric spaces.  
In this paper, we study basic third-order statements about continuous functions on such unrepresented compact metric spaces in Kohlenbach's higher-order Reverse Mathematics.  
We establish that some (very specific) statements are classified in the (second-order) Big Five of Reverse Mathematics, while most variations/generalisations are not provable from the latter, and much stronger systems.
Thus, continuous functions on unrepresented metric spaces are `wild', though `more tame' than (slightly) discontinuous functions on the reals.  
\end{abstract}

\setcounter{page}{0}
\tableofcontents
\thispagestyle{empty}
\newpage

\maketitle
\thispagestyle{empty}

\section{Introduction}
In a nutshell, we study basic third-order statements about continuous functions on `unrepresented' metric spaces, i.e.\ the latter come \emph{without} second-order representation, working in Kohlenbach's higher-order Reverse Mathematics, as introduced in \cite{kohlenbach2} and Section \ref{uk}.  We establish that certain (very specific) such statements are classified in the second-order Big Five of Reverse Mathematics, while most variations/generalisations are not provable from the latter, and much stronger systems.  Thus, we generalise the results in \cite{dagsamXIV} to metric spaces, but restrict ourselves to continuous functions.  

\smallskip

We believe these results to be of broad interest as the logic (and even mathematics) community should be aware of the influence representations have on some of the most basic objects, like continuous functions on compact metric spaces, that feature in undergraduate curricula in mathematics and physics.  

\smallskip

Moreover, our results also shed new light on Kohlenbach's \emph{proof mining} program: as stated in \cite{kohlenbach3}*{\S17.1} or \cite{kopo}*{\S1}, the success of proof mining often crucially depends on \emph{avoiding} the use of separability conditions.  By the results in this paper, avoiding such conditions seems to be a highly non-trivial affair. 

\smallskip

We provide some background and motivation for these results in Section \ref{back}, including a gentle introduction to higher-order Reverse Mathematics.  We formulate necessary definitions and axioms in Section~\ref{prel} and prove our main results in Section~\ref{soc}.  Finally, some foundational implications of our results, including related to the coding practise of Reverse Mathematics, are discussed in Section \ref{fouind}.

\subsection{Motivation and background}\label{back}
\subsubsection{Introduction}
In a nutshell, the topic of this paper is the study of compact metric spaces in higher-order arithmetic; this section provides ample motivation for this study, as well as a detailed overview of our results.  

\smallskip

Now, we assume familiarity with the program \emph{Reverse Mathematics}, abbreviated `RM' in the below.  
An introduction to RM for the mathematician-in-the-street may be found in \cite{stillebron}, while \cites{simpson2, damurm} are textbooks on RM.  
We shall work in Kohlenbach's higher-order RM, introduced in Section~\ref{uk}.    In Section \ref{self}, we provide some motivation for higher-order RM (and this paper), based on the following items.  
\begin{itemize}
\item[(a)] \emph{Simplicity}: the coding of continuous functions and other higher-order objects in second-order RM is generally not needed in higher-order RM. 
\item[(b)] \emph{Scope}: discontinuous $\R\di \R$-functions have been studied by Euler, Abel, Riemann, Fourier, and Dirichlet, i.e.\ the former are definitely part and parcel of ordinary mathematics.  
Discontinuous functions are studied directly in higher-order RM; the second-order approach via codes has its problems.  
\item[(c)] \emph{Generality}: do the results of RM, like the Big Five phenomenon, depend on the coding practise of RM?  Higher-order RM provides a (much needed, in our opinion) negative answer in the case of continuous functions.  
\end{itemize}
Having introduced and motivated higher-order RM in Sections \ref{uk}-\ref{self}, we discuss the results of this paper in some detail in Section \ref{tonz}.  
As we will see, the motivation for the study of compact metric spaces in this paper is provided by items (a) and (c) above.  In particular, we investigate whether the representation of metric spaces in second-order RM has an influence on the logical properties of third-order theorems about compact metric spaces.  The answer turns out to be positive, for all but one very specific choice of definitions.    

\subsubsection{Higher-order Reverse Mathematics}\label{uk}
We provide a gentle introduction to Kohlenbach's \emph{higher-order} RM, including the base theory $\RCAo$.
Our focus is on intuitive understanding rather than full technical detail.  

\smallskip

First of all, the language of $\RCAo$ includes all finite types.  In particular, the collection of \emph{all finite types} $\mathbf{T}$ is defined by the two clauses:
\begin{center}
(i) $0\in \mathbf{T}$   and   (ii)  If $\sigma, \tau\in \mathbf{T}$ then $( \sigma \di \tau) \in \mathbf{T}$,
\end{center}
where $0$ is the type of natural numbers, and $\sigma\di \tau$ is the type of mappings from objects of type $\sigma$ to objects of type $\tau$.
The following table provides an overview of the most common objects, their types, and their order. 
\begin{center}
\begin{tabular}{c|c|c|c}
Symbol & order & type & name \\
\hline
$n\in \N$ or $n^{0}$ & first & 0 & natural number \\ 
 $X\subset \N$ & second & 1 & subset of $\N$\\
 $f\in 2^{\N}$ & second & 1 & element of Cantor space $2^{\N}$ \\
 $f\in \N^{\N}$ or $f^{1}$ & second & 1 & element of Baire space  $\N^{\N}$ \\
 $Y: \N^{\N}\di \N$ or $Y^{2}$ & third & 2 & mapping of Baire space to $\N$ \\
 $Y: \N^{\N}\di \N^{\N}$ or $Y^{1\di 1}$ & third & 2 & mapping of Baire space to Baire space \\
\end{tabular}
\end{center}
One often identifies elements of Cantor space $2^{\N}$ and subsets of $\N$, as the former can be viewed as characteristic functions for the latter.  Similarly, a subset $X\subset \N^{\N}$ is given by the associated characteristic function $\mathbb{1}_{X}:\N^{\N}\di \{0,1\}$.
In this paper, we shall mostly restricts ourselves to third-order objects.     

\smallskip

Secondly, a basic axiom of mathematics is that functions \emph{map equal inputs to equal outputs}.
The \emph{axiom of function extensionality} guarantees this behaviour and is included in $\RCAo$.
As an example, we write $f=_{1}g$ in case $(\forall n\in \N)(f(n)=g(n) )$ and say that `$f$ and $g$ are equal elements of Baire space'.  
Any third-order $Y^{1\di 1}$ thus satisfies the following instance of the axiom of function extensionality:
\be\tag{$\textup{\textsf{E}}_{1\di 1}$}
(\forall f, g\in \N^{\N})(f=_{1}g\di Y(f)=_{1}Y(g)).
\ee
Now, the real number field $\R$ is central to analysis and other parts of mathematics.  
The real numbers are defined in $\RCAo$ in \emph{exactly the same way} as in $\RCA_{0}$, namely as fast-converging Cauchy sequences.  
Hence, the formulas `$x\in \R$', `$x<_{\R} y$', and `$x=_{\R}y$' have their usual well-known meaning.  
To define functions on the reals, we let an `$\R\di \R$-function' be any $Y^{1\di 1}$ that satisfies
\be\tag{$\textup{\textsf{E}}_{\R\di \R}$}\label{allizat}
(\forall x, y\in \R)(x=_{\R}y\di Y(x)=_{\R}Y(y)), 
\ee
which is the axiom of function extensionality relative to the (defined) equality `$=_{\R}$'.  
We stress that all symbols pertaining to the real numbers in \eqref{allizat} have their usual second-order meaning.  

\smallskip

Thirdly, $\RCA_{0}$ is a system of `computable mathematics' that includes comprehension for $\Delta_{1}^{0}$-formulas and induction for $\Sigma_{1}^{0}$-formulas.  
The former allows one to build algorithms, e.g.\ via primitive recursion, while the latter certifies their correctness.  
The base theory $\RCAo$ includes axioms that prove $\Delta_{1}^{0}$-comprehension and $\Sigma_{1}^{0}$-induction, as expected.  Moreover, $\RCAo$ includes the defining axiom of the recursor constant $\mathbf{R}_{0}$, namely that for $m\in \N$ and $f\in \N^{\N}$, we have:
\be\label{special}
\mathbf{R}_{0}(f, m, 0):= m \textup{ and } \mathbf{R}_{0}(f, m, n+1):= f(n, \mathbf{R}_{0}(f, m, n)), 
\ee
which defines primitive recursion with second-order functions.  
We hasten to add that higher-order parameters are allowed; as an example, we could use in \eqref{special} the function $f\in \N^{\N}$ defined as $f(n):= Y(q_{n})$ for any $Y:\R\di \N$ and where $(q_{n})_{n\in \N}$ lists all rational numbers.  The previous example also illustrates -to our mind- the need for \emph{lambda calculus} notation, where we would simply define $f\in \N^{\N}$ as $\lambda n^{0}.Y(q_{n})$, underscoring that $n$ is the (only) variable and $Y$ a parameter.  The system $\RCAo$ includes the defining axioms for $\lambda$-abstraction via combinators.  

\smallskip

Fourth, second-order RM includes many results about codes for continuous functions, and we would like to `upgrade' these results to third-order functions that satisfy the usual `epsilon-delta' definition of continuity.  To this end, $\RCAo$ includes the following fragment of the Axiom of Choice, provable in $\ZF$:
\be\tag{$\QFAC^{1,0}$}
(\forall f\in \N^{\N})(\exists n^{0}\in \N)A(f, n)\di (\exists Y:\N^{\N}\di \N)(\forall f\in \N^{\N})A(f, Y(f)),
\ee
for any quantifier-free formula $A$.  Now, the following formula 
\be\label{walt}
\text{\emph{$\Phi\subset \N$ is a total code for a continuous $\N^{\N}\di \N^{\N}$-function}}
\ee
has exactly the form as in the antecedent of $\QFAC^{0,1}$.  
Applying the latter to \eqref{walt}, we readily obtain a function $Z^{1\di 1}$ such that $Z(f)$ equals the value of $\Phi$ at any $f\in \N^{\N}$.  
A similar argument goes through for codes of continuous $\R\di \R$-functions.  

\smallskip

Finally, $\RCAo$ and $\RCA_{0}$ prove the same second-order sentences, a fact that is proved via the so-called ECF-translation (see \cite{kohlenbach2}*{\S2}).  
In a nutshell, the latter replaces third- and higher-order objects by second-order codes for continuous functions, a concept not alien to second-order RM.  
We will discuss the coding of continuous functions in more detail in Section \ref{self}. 

\subsubsection{The coding practise of Reverse Mathematics}\label{self}
We discuss the coding practise of RM, which will be seen to provide motivation for Kohlenbach's higher-order RM.

\smallskip

First of all, second-order RM makes use of the rather frugal language of second-order arithmetic, which only includes variables for natural numbers $n\in \N$ and sets of natural numbers $X\subset \N$.
As a result, higher-order objects like functions on the reals and metric spaces, are unavailable and need to be `represented' or `coded' by second-order objects.          
This coding practise can complicate basic definitions: the reader need only compare the one-line `epsilon-delta' definition of continuity to the second-order definition 
in \cite{simpson2}*{II.6.1}; the latter takes the better half of a page.  Similar complications arise for metric spaces, where the reader can compare Definition \ref{donkc} below to \cite{simpson2}*{I.8.6}.   
Hence, a framework that includes third- and higher-order objects provides a \emph{simpler} approach.  

\smallskip

Secondly, discontinuous functions have been studied in mathematics long before the advent of set theory, by `big name' mathematicians like Euler, Dirichlet, Riemann, Abel, and Fourier, as discussed in \cite{samrep}*{\S5.2}. 
Hence, basic properties of discontinuous functions are part of ordinary mathematics and should be studied in RM.  
Higher-order RM provides a natural framework for the study of discontinuous functions, as explored in detail in \cites{kohlenbach2, dagsamXIV}.
By contrast, the study of discontinuous functions via codes is problematic\footnote{In a nutshell, $\RCA_{0}$ and $\RCAo$ are consistent with \emph{Church's thesis} $\textsf{CT}$, i.e.\ the statement \emph{all real numbers are computable}.  
In fact, the recursive sets form the minimal $\omega$-model of $\RCA_{0}$ by \cite{simpson2}*{I.7.5}.  Now, $\RCAo+\textsf{CT}$ proves that all functions on the reals are continuous, while in $\RCA_{0}+\textsf{CT}$, there are plenty of codes for discontinuous functions, e.g.\ for the Heaviside function.  Thus, $\RCAo$ exhibits a connection between second- and third-order objects and theorems, which is seemingly obliterated by coding functions.}, as discussed in detail in \cite{samrep}*{\S6.2.2}.

\smallskip

Thirdly, a central objective of mathematical logic is the classification of mathematical statements in hierarchies according to their logical strength.  
An example due to Simpson is the \emph{G\"odel hierarchy} from \cite{sigohi}. 
The goal of RM, namely finding the minimal axioms that prove a theorem of ordinary mathematics, fits squarely into this objective.  Ideally, the 
place of a given statement in the hierarchy at hand does not depend greatly on the representation used.  In particular, RM seeks to analyse theorems of ordinary mathematics `as they stand'.
The following quotes from \cite{simpson2}*{p.\ 32 and 137} illustrate this claim.  
\begin{quote}
The typical constructivist response to a nonconstructive mathematical theorem is to modify the theorem by adding hypotheses or ``extra data''.
In contrast, our approach in this book is to analyze the provability of mathematical theorems as they stand [\dots]
\end{quote}
\begin{quote}
[\dots] we seek to draw out the set existence
assumptions which are implicit in the ordinary mathematical theorems \emph{as they stand}.
\end{quote}
Essentially the same claim may be found in \cite{damurm}*{\S10.5.2} and in many parts of the RM literature.  The main point is always the same: RM ideally studies mathematics `as is' without any logical additions.   

\smallskip

Since the coding of continuous functions is conspicuously absent from mainstream mathematics, it is then a natural question whether the aformentioned coding has an influence 
on the classification of theorems in RM.  The following theorem implies that at least the Big Five phenomenon of RM does not really depend on the coding of continuous functions on various spaces.
\begin{thm}\label{floppi}
The system $\RCAo$ proves the following for $\mathbb{X}=\N^{\N}$ or $\mathbb{X}=\R$.
\begin{center}
\emph{Let $\Phi$ be a code for a continuous $\mathbb{X}\di \mathbb{X}$-function.  There is a third-order $F:\mathbb{X}\di \mathbb{X}$ such that $F(x)$ equals the value of $\Phi$ at $x$ for any $x\in \mathbb{X}$.}
\end{center}
The system $\RCAo+\WKL_{0}$ proves the following for $\mathbb{X}=2^{\N}$ or $\mathbb{X}=[0,1]$.
\begin{center}
\emph{Let the third-order function $F:\mathbb{X}\di \mathbb{X}$ be continuous on $\mathbb{X}$.  Then there is a code $\Phi$ such that $F(x)$ equals the value of $\Phi$ at any $x\in \mathbb{X}$.}
\end{center}
\end{thm}
\begin{proof}
See \cite{kohlenbach4}*{\S4} and \cite{dagsamXIV}*{\S2} for proofs. 
\end{proof}
As a corollary, we observe that over $\RCAo$, the \emph{second-order} axiom $\WKL_{0}$ is equivalent to \emph{third-order} theorems like \emph{for any third-order $F:[0,1]\di \R$, continuity implies boundedness}.  One could argue that continuous functions are `really' second-order, but the reader should keep the previous sentence in mind nonetheless.

\subsubsection{Metric spaces and higher-order Reverse Mathematics}\label{tonz}
Having introduced and motivated higher-order RM in the previous two sections, we can now discuss and motivate the results in this paper in some detail. 

\smallskip

First of all, an immediate corollary of Theorem \ref{floppi} is that over $\RCAo$, the \emph{second-order} axiom $\WKL_{0}$ is equivalent to many basic \emph{third-order} theorems from real analysis about continuous functions.  Motivated by this observation, Dag Normann and the author show in \cite{samBIG3,dagsamXIV, samBIG4} that \textbf{many} \emph{third-order} theorems from real analysis about (possibly) {discontinuous} functions on the reals,
are equivalent to the \emph{second-order} Big Five, over $\RCAo$. 
Moreover, \emph{slight} variations/generalisations of the function class at hand yield third-order theorems that are not provable from the Big Five \emph{and} the same for much stronger systems like $\Z_{2}^{\omega}+\QFAC^{0,1}$ from Section~\ref{prel}.  It is then a natural question whether a similar phenomenon can be found in other parts of mathematics and RM.

\smallskip

Secondly, in this paper, we study a different kind of generalisation than the one in \cite{dagsamXIV}: rather than going beyond the continuous functions, we study properties of the latter on compact metric spaces.  
Now, the study of the latter in second-order RM proceeds via codes: a \emph{complete separable metric space} is represented via a countable and dense subset, as can be gleaned from \cite{simpson2}*{II.5.1} or \cite{browner}.  By contrast, we use the standard textbook definition of metric space as in Definition \ref{donkc} without any additional data except that we are dealing with sets of reals.  This study is not just \emph{spielerei} as \emph{avoiding} separability is e.g.\ important in proof mining, as follows. 
\begin{quote}  
[\dots] it is crucial to exploit the fact that the proof to be analyzed does not use any separability assumption on the underlying spaces [\dots]. (\cite{kopo}*{\S1})
\end{quote}
\begin{quote}
It will turn out that for [the aforementioned uniformity conditions] to hold we -in particular- must not use any separability assumptions on the spaces. (\cite{kohlenbach3}*{p.\ 377})
\end{quote}
Thirdly, in light of the previous two paragraphs, it is then a natural question whether basic properties of compact metric spaces \emph{without separability conditions} are provable from second-order (comprehension) axioms or not.
Theorem \ref{klon} provides a (rather) negative answer: well-known theorems due to Ascoli, Arzel\`a, Dini, Heine, and Pincherle, formulated for metric spaces, are not provable in $\Z_{2}^{\omega}$, a conservative extension of $\Z_{2}$ introduced in Section \ref{prel}.  We only study metric spaces $(M, d)$ where $M$ is a subset of the reals or Baire space, i.e.\ the metric $d:M^{2}\di \R$ is just a third-order mapping.  
By contrast, some (very specific) basic properties of metric spaces are provable from the Big Five and related systems by Theorem \ref{typemon}.  
 
\smallskip

Fourth, the negative results in this paper are established using the uncountability of $\R$ as formalised by the following principles (see Section \ref{prel} for details).  
\begin{itemize}
\item $\NIN_{[0,1]}$: there is no injection from $[0,1]$ to $\N$.  
\item $\NBI_{[0,1]}$: there is no bijection from $[0,1]$ to $\N$.
\end{itemize}
In particular, these principles are not provable in relatively strong systems, like $\Z_{2}^{\omega}$ from Section \ref{prel}.  
In Section~\ref{soco}, we identify a long and robust list of theorems that imply $\NBI_{[0,1]}$ or $\NIN_{[0,1]}$.  
We have shown in \cite{dagsamX, dagsamXIV, samhabil} that many third-order theorems imply $\NIN_{[0,1]}$ while we only know few theorems that only imply $\NBI_{[0,1]}$.  
As will become clear in Section~\ref{refi}, metric spaces provide (many) elegant examples of the latter.  
We also refine our results in Section \ref{refi}, including connections to the RM of weak K\"onig's lemma and the Jordan decomposition theorem.  

\smallskip

In conclusion, we show that many basic (third-order) properties of continuous functions on metric spaces cannot be proved from second-order (comprehension) axioms 
when we omit the second-order representation of these spaces.   A central principle is the uncountability of the reals as formalised by $\NBI_{[0,1]}$ introduced above.  
These results carry foundational implications, as discussed in Section \ref{fouind}.

\subsection{Preliminaries and definitions}\label{prel}
We introduce some definitions, like the notion of open set or metric space in RM, and axioms that cannot be found in \cite{kohlenbach2}.  We emphasise that we only study metric spaces $(M, d)$ where $M$ is a subset of $\N^{\N}$ or $\R$, modulo the coding of finite sequences\footnote{We use $w^{1^{*}}$ to denote finite sequences of elements of $\N^{\N}$ and $|w|$ as the length of $w^{1^{*}}$.} of reals.  Thus, everything can be formalised in the language of third-order arithmetic, i.e.\ we do not really go much beyond analysis on the reals. 

\smallskip

Zeroth of all, we need to define the notion of (open) set.
Now, open sets are represented in second-order RM by \emph{countable unions of basic open balls}, namely as in \cite{simpson2}*{II.5.6}.  
In light of \cite{simpson2}*{II.7.1}, \emph{\(codes for\) continuous functions} provide an equivalent representation over $\RCA_{0}$. 
In particular, the latter second-order representation is exactly the following definition restricted to (codes for) continuous functions, as can be found in \cite{simpson2}*{II.6.1}.  
\bdefi\label{bchar}~
\begin{itemize}
\item A \emph{set} $U\subset \R$ \(and its complement $U^{c}$\) is given by $h_{U}:\R\di [0,1]$ where we say `$x\in U$' if and only if $h_{U}(x)>0$.  
\item A set $U\subset \R$ is \emph{open} if $y\in U$ implies $(\exists N\in \N)(\forall z\in B(y, \frac{1}{2^{N}})(z\in U )$.  
A set is closed if the complement is open. 
\item A set $U\subset \R$ is \emph{finite} if there is $N\in \N$ such that for any finite sequence $(x_{0}, \dots, x_{N})$, there is $i\leq N$ with $x_{i}\not \in U$.  
We sometimes write `$|U|\leq N$'.

\end{itemize}
\edefi
Now, codes for continuous functions denote third-order functions in $\RCAo$ by \cite{dagsamXIV}*{\S2}, i.e.\ Def.\ \ref{bchar} thus includes the second-order definition of open set.  To be absolutely clear, combining \cite{dagsamXIV}*{Theorem 2.2} and \cite{simpson2}*{II.7.1}, $\RCAo$ proves
\begin{center}
\emph{\textup{[}a second-order code $U$ for an open set\textup{]} represents an open set as in Def.~\ref{bchar}}. 
\end{center}
Assuming Kleene's quantifier $(\exists^{2})$ defined below, Def.\ \ref{bchar} is equivalent to the existence of a characteristic function for $U$; the latter definition is used in e.g.\ \cite{dagsamVII, dagsamXV}.  
The interested reader can verify that over $\RCAo$, a set $U$ as in Def.\ \ref{bchar} is open if and only if $h_{U}$ is lower semi-continuous. 

\smallskip

First of all, we shall study metric spaces $(M, d)$ as in Definition~\ref{donkc}, where $M$ comes with its own equivalence relation `$=_{M}$' and the metric $d$ satisfies 
the axiom of extensionality on $M$ as follows
\[
(\forall x, y, v, w\in M)\big([x=_{M}y\wedge v=_{M}w]\di d(x, v)=_{\R}d(y, w)\big).
\]
Similarly, we use $F:M\di \R$ to denote functions from $M$ to $\R$; the latter satisfy 
\be\tag{\textup{\textsf{E}}$_{M}$}
(\forall x, y\in M)(x=_{M}y\di F(x)=_{\R}F(y)), 
\ee
i.e.\ the axiom of function extensionality relative to $M$.
\bdefi\label{donkc}
A functional $d: M^{2}\di \R$ is a \emph{metric on $M$} if it satisfies the following properties for $x, y, z\in M$:
\begin{enumerate}
 \renewcommand{\theenumi}{\alph{enumi}}
\item $d(x, y)=_{\R}0 \asa  x=_{M}y$,
\item $0\leq_{\R} d(x, y)=_{\R}d(y, x), $
\item $d(x, y)\leq_{\R} d(x, z)+ d(z, y)$.
\end{enumerate}
We use standard notation like $B_{d}^{M}(x, r)$ to denote $\{y\in M: d(x, y)<r\}$.
\edefi
To be absolutely clear, quantifying over $M$ amounts to quantifying over $\N^{\N}$ or $\R$, perhaps modulo coding, i.e.\ the previous definition can be made in third-order arithmetic for the intents and purposes of this paper.  The definitions of `open set in a metric space' and related constructs are now clear \emph{mutatis mutandis}.

\smallskip

Secondly, the following definitions are now standard, where we note that the first item is called `Heine-Borel compact' in e.g.\ \cite{browner, brownphd}. 
Moreover, coded complete separable metric spaces as in \cite{simpson2}*{I.8.2} are only \emph{weakly complete} over $\RCA_{0}$.  
\bdefi[Compactness and around] For a metric space $(M, d)$, we say that
\begin{itemize}
\item $(M, d)$ is \emph{countably-compact} if for any $(a_{n})_{n\in \N}$ in $M$ and sequence of rationals $(r_{n})_{n\in \N}$ such that we have $M\subset \cup_{n\in \N}B^{M}_{d}(a_{n}, r_{n})$, there is $m\in \N$ such that  $M\subset \cup_{n\leq m}B^{M}_{d}(a_{n}, r_{n})$,
\item $(M, d)$ is \emph{strongly countably-compact} if for any sequence $(O_{n})_{n\in \N}$ of open sets in $M$ such that $M\subset \cup_{n\in \N}O_{n}$, there is $m\in \N$ such that  $M\subset \cup_{n\leq m}O_{n}$,
\item $(M, d)$ is \emph{compact} in case for any $\Psi:M\di \R^{+}$, there are $x_{0}, \dots, x_{k}\in M$ such that $\cup_{i\leq k}B_{d}^{M}(x_{i}, \Psi(x_{i}))$ covers $M$, 
\item $(M, d)$ is \emph{sequentially compact} if any sequence has a convergent sub-sequence,
\item $(M, d)$ is \emph{limit point compact} if any infinite set in $M$ has a limit point,
\item $(M, d)$ is \emph{complete} in case every \emph{Cauchy}$^{\ref{couchy}}$ sequence converges,
\item $(M, d)$ is \emph{weakly complete} if every \emph{effectively}\footnote{A sequence $(w_{n})_{n\in \N}$ in $(M, d)$ is \emph{Cauchy} if $(\forall k\in \N)(\exists N\in \N)(\forall m, n\geq N)( d(w_{n}, w_{m})<\frac{1}{2^{k}})$. A sequence is \emph{effectively Cauchy} if there is $g\in \N^{\N}$ such that $g(k)=N$ in the previous formula.\label{couchy}} Cauchy sequence converges,
\item $(M, d)$ is \emph{totally bounded} if for all $k\in \N$, there are $w_{0}, \dots, w_{m}\in \N$ such that $\cup_{i\leq m}B_{d}^{M}(w_{i}, \frac{1}{2^{k}})$ covers $M$.
\item $(M, d)$ is \emph{effectively totally bounded} if there is a sequence of finite sequences $(w_{n})_{n\in \N}$ in $M$ such that for all $k\in \N$ and $x\in M$, there is $i<|w_{k}|$ such that $x\in B(w(i), \frac{1}{2^{k}})$.
\item a set $C\subset M$ is \emph{sequentially closed} if for any sequence $(w_{n})_{n\in \N}$ in $C$ converging to $w\in M$, we have $w\in C$.  
\item $(M, d)$ has the \emph{Cantor intersection property} if any sequence of nonempty closed sets with $M\supseteq C_{0}\supseteq\dots \supseteq C_{n}\supseteq C_{n+1}$, has a nonempty intersection,
\item $(M, d)$ has the \emph{sequential Cantor intersection property} if the sets in the previous item are sequentially closed. 
\item $(M, d)$ is \emph{separable} if there is a sequence $(x_{n})_{n\in \N}$ in $M$ such that $(\forall x\in M, k\in \N)(\exists n\in \N)( d(x, x_{n})<\frac{1}{2^{k}})$.
\end{itemize}
\edefi
Thirdly, full second-order arithmetic $\Z_{2}$ is the `upper limit' of second-order RM.  The systems $\Z_{2}^{\omega}$ and $\Z_{2}^{\Omega}$ are conservative extensions of $\Z_{2}$ by \cite{hunterphd}*{Cor.\ 2.6}. 
The system $\Z_{2}^{\Omega}$ is $\RCAo$ plus Kleene's quantifier $(\exists^{3})$ (see e.g.\ \cite{dagsamXIV, hunterphd}), while $\Z_{2}^{\omega}$ is $\RCAo$ plus $(\SS_{k}^{2})$ for every $k\geq 1$; the latter axiom states the existence of a functional $\SS_{k}^{2}$ deciding $\Pi_{k}^{1}$-formulas in Kleene normal form.  
The system $\FIVE^{\omega}\equiv \RCAo+(\SS_{1}^{2})$ is a $\Pi_{3}^{1}$-conservative extension of $\FIVE$ (\cite{yamayamaharehare}), where $\SS_{1}^{2}$ is also called the Suslin functional. 
We also write $\ACAo$ for $\RCAo+(\exists^{2})$ where the latter is as follows 
\be\label{muk}\tag{$\exists^{2}$}
(\exists E:\N^{\N}\di \{0,1\})(\forall f \in \N^{\N})\big[(\exists n\in \N)(f(n)=0) \asa E(f)=0    \big]. 
\ee
Over $\RCAo$, $(\exists^{2})$ is equivalent to the existence of Feferman's $\mu$ (see \cite{kohlenbach2}*{Prop.\ 3.9}), defined as follows for all $f\in \N^{\N}$:
\[
\mu(f):= 
\begin{cases}
n & \textup{if $n$ is the least natural such that $f(n)=0$, }\\
0 & \textup{ if $f(n)>0$ for all $n\in \N$}
\end{cases}.
\]
Fourth, the uncountability of the reals, formulated as follows, is studied in \cite{dagsamX}.
\begin{itemize}
\item $\NIN_{[0,1]}$: there is no $Y:[0,1]\di \N$ that is injective\footnote{A function $f:X\di Y$ is injective if different $x, x'\in X$ yield different $f(x), f(x')\in Y$.}.
\item $\NBI_{[0,1]}$: there is no $Y:[0,1]\di \N$ that is both injective and surjective\footnote{A function $f:X\di Y$ is surjective if for every $y\in Y$, there is $x\in X$ with $f(x)=_{Y}y$.}. 
\end{itemize}
It is shown in \cite{dagsamXI, dagsamX} that $\Z_{2}^{\omega}$ cannot prove $\NBI_{[0,1]}$ and that $\Z_{2}^{\omega}+\QFAC^{0,1}$ cannot prove $\NIN_{[0,1]}$, where the latter 
is countable choice\footnote{To be absolutely clear, $\QFAC^{0,1}$ states that for every $Y^{2}$, $(\forall n\in \N)(\exists f\in \N^{\N})(Y(f, n)=0)$ implies $(\exists \Phi^{0\di 1})(\forall n\in \N)(Y(\Phi(n), n)=0)$.} for quantifier-free formulas.   Moreover, many third-order theorems imply $\NIN_{[0,1]}$, as also established in \cite{dagsamX}.
By contrast, that $\R$ cannot be enumerated is formalised by Theorem \ref{cant2}.
\begin{thm}\label{cant2}
For any sequence of distinct real numbers $(x_{n})_{n\in \N}$ and any interval $[a,b]$, there is $y\in [a,b]$ such that $y$ is different from $x_{n}$ for all $n\in \N$.
\end{thm}
The previous theorem is rather tame, especially compared to $\NIN_{[0,1]}$.
Indeed, \cite{grayk} includes an efficient computer program that computes the number $y$ from Theorem \ref{cant2} in terms of the other data; a proof of Theorem \ref{cant2} in $\RCA_{0}$ can be found in \cite{simpson2}*{II.4.9}, while a proof in Bishop's \emph{Constructive Analysis} is found in \cite{bish1}*{p.\ 25}.  

\smallskip

Finally, the following remark discusses an interesting aspect of $(\exists^{2})$ and $\NIN_{[0,1]}$.
\begin{rem}[On excluded middle]\label{LEM}\rm
Despite the grand stories told in mathematics and logic about Hilbert and the law of excluded middle, 
the `full' use of the latter law in RM is almost somewhat of a novelty.  To be more precise, the law of excluded middle as in $(\exists^{2})\vee \neg(\exists^{2})$ is extremely useful, namely as follows:   
suppose we are proving $T\di \NIN_{[0,1]}$ over $\RCAo$.  
Now, in case $\neg(\exists^{2})$, all functions on $\R$ (and $\N^{\N}$) are continuous by \cite{kohlenbach2}*{Prop.\ 3.12}.
Clearly, any continuous $Y:[0,1]\di \N$ is not injective, i.e.\ $\NIN_{[0,1]}$ follows in the case that $\neg(\exists^{2})$.  
Hence, what remains is to establish $T\di \NIN_{[0,1]}$ \emph{in case we have} $(\exists^{2})$.  However, the latter axiom e.g.\ implies $\ACA_{0}$ (and sequential compactness) and can uniformly convert reals to their binary representations, which can simplify the remainder of the proof 

\smallskip

Here, $\NIN_{[0,1]}$ is just one example and there are many more, all pointing to a more general phenomenon: while invoking $(\exists^{2})\vee \neg(\exists^{2})$ may be non-constructive, it does lead to 
a short proof via case distinction:  in case $(\exists^{2})$, one has access to a stronger system while in case $\neg(\exists^{2})$, the theorem at hand is a triviality (like for $\NIN_{[0,1]}$ in the previous paragraph), or at least has a well-known second-order proof.  
We can also work over $\RCAo+\WKL_{0}$, noting that the latter establishes that continuous functions on $[0,1]$ or $2^{\N}$ have RM-codes (see \cite{dagsamXIV}*{\S2} and \cite{kohlenbach4}*{\S4}).
\end{rem}

\section{Analysis on unrepresented metric spaces}\label{soc}
We show that some (very specific) properties of continuous functions on compact metric spaces are classified in the (second-order) Big Five systems of Reverse Mathematics (Section \ref{refi}), while most variations/generalisations are not provable from the latter, and much stronger systems (Section~\ref{soco}).  
The negative results are (mostly) established by deriving $\NBI_{[0,1]}$ (Theorem \ref{klon}), which is not provable in $\Z_{2}^{\omega}$. 
We also show that $\NIN_{[0,1]}$ does not follow in most cases (Theorem \ref{zlonk}). 

\subsection{Obtaining the uncountability of the reals}\label{soco}
In this section, we show that basic properties of continuous functions on compact metric spaces, like Heine's theorem in item \eqref{itemb}, imply the uncountability of the reals as in $\NBI_{[0,1]}$.  These basic properties are therefore not provable in $\Z_{2}^{\omega}$.  

\smallskip

First of all, fragments of the induction axiom are sometimes used in an essential way in second-order RM (see e.g.\ \cite{neeman}).  
The equivalence between induction and bounded comprehension is also well-known in second-order RM (\cite{simpson2}*{X.4.4}).
We seem to need a little bit of the induction axiom as follows.
\begin{princ}[$\IND_{1}$]\label{IX}
Let $Y^{2}$ satisfy $(\forall n \in \N) (\exists !f \in 2^{\N})[Y(n,f)=0]$.  Then $ (\forall n\in \N)(\exists w^{1^{*}})\big[ |w|=n\wedge  (\forall i < n)(Y(i,w(i))=0)\big]$.  
\end{princ}
Note that $\IND_{1}$ is a special case of the axiom of finite choice, and is valid in all models considered in \cites{dagsam, dagsamII, dagsamIII, dagsamV, dagsamVI, dagsamVII, dagsamIX, dagsamX}, i.e.\ $\Z_{2}^{\omega}+\IND_{1}$ cannot prove $\NBI_{[0,1]}$.
We have (first) used $\IND_{1}$ in the RM of the Jordan decomposition theorem in \cite{dagsamXI}.

\smallskip

Secondly, the items in Theorem \ref{klon} are essentially those in \cite{browner}*{Theorem 4.1} or \cite{simpson2}*{IV.2.2}, but without codes.  
Equivalences of certain (coded) definitions of compactness are studied in second-order RM in e.g.\ \cite{browner2, brownphd}.  
\begin{thm}[$\RCAo+\IND_{1}$]\label{klon}
The principle $\NBI_{[0,1]}$ follows from any of the items \eqref{itema}-\eqref{itemlast} where $(M, d)$ is a metric space with $M\subset \R$.
\begin{enumerate}
\renewcommand{\theenumi}{\alph{enumi}}
\item For countably-compact $(M, d)$ and sequentially continuous $F:M\di \R$, $F$ is bounded on $M$.\label{itema}
\item Item \eqref{itema} with `bounded' replaced by `uniformly continuous'.\label{itemb}
\item Item \eqref{itema} with `bounded' replaced by `has a supremum'.\label{itemc}
\item Item \eqref{itema} with `bounded' replaced by `attains a maximum'.\label{itemd}
\item A countably-compact $(M, d)$ has the sequential Cantor intersection property.\label{iteme}
\item A countably-compact metric space $(M, d)$ is separable. \label{itemf}
\end{enumerate}
The previous items still imply $\NBI_{[0,1]}$ if we replace `countably-compact' by `compact' or `\(weakly\) complete and totally bounded' or `strongly countably-compact'.
\begin{enumerate}
\setcounter{enumi}{7}
\renewcommand{\theenumi}{\alph{enumi}}
\item For sequentially compact $(M, d)$, any continuous $F:M\di \R$ is bounded.\label{itemg}
\item Item \eqref{itemg} with `bounded' replaced by `uniformly continuous'.\label{itemgg}
\item Item \eqref{itemg} with `bounded' replaced by `has a supremum'.\label{itemh}
\item Item \eqref{itemg} with `bounded' replaced by `attains a maximum'.\label{itemi}
\item Items \eqref{itemg}-\eqref{itemi} assuming a modulus of continuity. \label{modulus}
\item Dini's theorem \(\cite{diniberg2, dinipi2, dinipi}\).  Let $(M, d)$ be sequentially compact and let $F_{n}: (M\times \N)\di \R$ be a monotone sequence of continuous functions converging to continuous $F:M\di \R$.  
Then the convergence is uniform.   \label{diniitem}
\item On a sequentially compact metric space $(M, d)$, equicontinuity implies uniform equicontinuity.  \label{equi}
\item \(Pincherle, \cite{tepelpinch}*{p.\ 67}\). For sequentially compact $(M, d)$ and continuous $F:M\di \R^{+}$, we have  $(\exists k\in \N)(\forall w\in M)(F(w)> \frac{1}{2^{k}} )$.\label{pinker}
\item \(Ascoli-Arzel\`a, \cite{simpson2}*{III.2}\).  For sequentially compact $(M, d)$, a uniformly bounded and equicontinuous sequence of functions on $M$ has a uniformly
convergent sub-sequence.\label{AA}
\item Any sequentially compact $(M, d)$ is strongly countably-compact. \label{itemlast2}
\item Any sequentially compact $(M, d)$ is separable. \label{itemii}
\item Any sequentially compact $(M, d)$ has the seq.\ Cantor intersection property.\label{itemlast}
\item A sequentually compact metric space $(M, d)$ is limit point compact. \label{itemiii}
\end{enumerate}
Items \eqref{itemg}-\eqref{modulus} are provable in $\Z_{2}^{\Omega}$ \(via the textbook proof\).  
\end{thm}
\begin{proof}
First of all, by Remark \ref{LEM}, we may assume $(\exists^{2})$ as $\NBI_{[0,1]}$ is trivial in case $\neg(\exists^{2})$. 
Now suppose $Y:[0,1]\di \N$ is a bijection, i.e.\ injective and surjective. Define $M$ as the union of the new symbol $\{0_{M}\}$ and the set $N:=\{ w^{1^{*}}: (\forall i<|w|)(Y(w(i))=i)\}$. 
We define `$=_{M}$' as $0_{M}=_{M}0_{M}$, $u\ne_{M} 0_{M}$ for $u\in N$, and $w=_{M}v$ if $w=_{1^{*}}v$ and $w, v\in N$.
The metric $d:M^{2}\di \R$ is defined as $d(0_{M}, 0_{M})=_{\R}0$, $d(0_{M}, u)=d(u, 0_{M})=\frac{1}{2^{|u|}}$ for $u\in N$ and $d(w, v)=|\frac{1}{2^{|v|}}-\frac{1}{2^{|w|}}|$ for $w, v\in N$.
Since $Y$ is an injection, we have $d(v, w)=_{\R}0 \asa  v=_{M}w$.  The other properties of a metric space from Definition \ref{donkc} follow by definition (and the triangle equality of the absolute value on the reals).  

\smallskip

Secondly, to show that $(M, d)$ is countably-compact, fix a sequence $(a_{n})_{n\in \N}$ in $M$ and a sequence of rationals $(r_{n})_{n\in \N}$ such that we have $M\subset \cup_{n\in \N}B^{M}_{d}(a_{n}, r_{n})$
Suppose $0_{M}\in B_{d}^{M}(a_{n_{0}}, r_{n_{0}})$ for $a_{n_{0}}\ne_{M}0_{M}$, i.e.\ $\frac{1}{2^{|a_{n_{0}}|}}=d(0_{M}, a_{n_{0}})<r_{n_{0}}$. 
Then $|\frac{1}{2^{|y|}}-\frac{1}{2^{|a_{n_{0}}|}}| =d(y, a_{n_{0}})<r_{n_{0}}$ holds for all $y\in N$ such that $|y|>|a_{n_{0}}|$.
Now use $\IND_{1}$ to enumerate the (finitely many) reals $z\in M$ with $|z|<|a_{n_{0}}|$.  
In this way, there exists a finite sub-covering of $\cup_{n\in \N}B^{M}_{d}(a_{n}, r_{n})$ of at most $|a_{n_{0}}|+1$ elements.
The proof is analogous (and easier) in case $a_{n_{0}}=_{M}0_{M}$.  
Thus, $(M, d)$ is a countably-compact metric space.  

\smallskip
  
Thirdly, define the function $F: M\di \R$ as follows: $F(0_{M}):= 0$ and $F(w):=|w|$ for any $w\in N$.  
Clearly, if the sequence $(w_{n})_{n\in \N}$ in $M$ converges to $0_{M}$, either it is eventually constant $0_{M}$ or lists all reals in $[0,1]$.  The latter case is impossible by Theorem \ref{cant2}.  
Hence, $F$ is sequentially continuous at $0_{M}$, but not continuous at $0_{M}$.  To show that $F$ is (sequentially) continuous at $w\ne 0_{M}$, consider the formula $|\frac{1}{2^{|w|}}-\frac{1}{2^{|v|}}|=d(v, w)<\frac{1}{2^{N}}$; the latter is false for $N\geq |w|+2$ and any $v \ne_{M}0_{M}$.  Thus, the following formula is (vacuously) true: 
\[\textstyle
(\forall k\in \N)(\exists N\in \N)(\forall v\in B_{d}^{M}(w, \frac{1}{2^{N}}))(|F(w)-F(v)|< \frac{1}{2^{k}}).
\]
i.e.\ $F$ is continuous at $w\ne_{M}0_{M}$, with a (kind of) modulus of continuity given.
Applying item \eqref{itema} (or item \eqref{itemc}-\eqref{itemd}), we obtain a contradiction as $F$ is clearly unbounded on $M$.
This contradiction yields $\NBI_{[0,1]}$ and the same for item \eqref{itemb} as $F$ is not (uniformly) continuous.  

\smallskip

Fourth, to obtain $\NBI_{[0,1]}$ from item \eqref{iteme}, suppose again the former is false and $Y:[0,1]\di \R$ and $(M, d)$ are as above.  
Define $C_{n}:= \{x\in N: |x|> n+1 \}$ and note that this set is non-empty (as $Y$ is a surjection) but satisfies $\cap_{n} C_{n}=\emptyset$.
Item~\eqref{iteme} now yields a contradiction if we can show that $C_{n}$ is sequentially closed.  To the latter end,  let $(w_{k})_{k\in \N}$ be a sequence in $C_{n}$ with limit $w\in M$.  
In case $w=_{M}0_{M}$, we make the same observation as in the third paragraph: either the sequence $(w_{k})_{k\in \N}$ is eventually constant $0_{M}$ or enumerates the reals in $[0,1]$.  Both are impossible, i.e.\ this case does not occur.  In case $w\ne 0_{M}$, we have 
\[\textstyle
(\forall k\in \N)(\exists N\in \N)(\forall n\geq N)(|\frac{1}{2^{|w|}}-\frac{1}{2^{|w_{n}|}}| =d(w, w_{n})<\frac{1}{2^{k}}), 
\]
which is only possible if $(w_{n})_{n\in \N}$ is eventually constant $w$.  In this case of course, $w\in C_{n}$, i.e.\ $C_{n}$ is sequentially closed, and $\eqref{iteme}\di\NBI_{[0,1]}$ follows. 
Regarding item~\eqref{itemf}, suppose $(M, d)$ is separable, i.e.\ there is a sequence $(w_{n})_{n\in \N}$ such that 
\be\label{tiops}\textstyle
(\forall w\in M, k\in \N)(\exists n\in \N)( |\frac{1}{2^{|w|}}-\frac{1}{2^{|w_{n}|}}|=d(w, w_{n})<\frac{1}{2^{k}}).
\ee
As in the above, for $w\ne_{M} 0_{M}$ and $k_{0}=|w|+2$, the formula $d(w, w_{n})<\frac{1}{2^{k_{0}}}$ is false for any $n\in \N$, i.e.\ we also obtain a contradiction in this case, yielding $\NBI_{[0,1]}$. 

\smallskip

Fifth, for the sentences between items \eqref{itemf} and \eqref{itemg}, $(M, d)$ is also complete and (strongly countably) compact, which is proved in (exactly) the same way as in the second paragraph: any ball around $0_{M}$ covers `most' of $M$; to show that $(M, d)$ is complete, let $(w_{n})_{n\in \N}$ be a Cauchy sequence, i.e.\ we have
\[\textstyle
(\forall k\in \N)(\exists N\in \N)(\forall n, m\geq N)( d(w_{n}, w_{m})<\frac{1}{2^{k}}).  
\]
Then $(w_{n})_{n\in \N}$ is either eventually constant or enumerates all reals in $[0,1]$.  
The latter is impossible by Theorem \ref{cant2}, i.e.\ $(w_{n})_{n\in \N}$ converges to some $w\in M$.
Note that a continuous function is trivially sequentially continuous.  

\smallskip

Sixth, to obtain $\NBI_{[0,1]}$ from item \eqref{itemg} and higher, recall the set $N:=\{ w^{1^{*}}: (\forall i<|w|)(Y(w(i))=i)\}$ and consider $(N, d)$, which is a metric space in the same way as for $(M, d)$.
To show that $(N, d)$ is sequentially compact, let $(w_{n})_{n\in \N}$ be a sequence in $N$.  In case $(\forall n\in \N)( |w_{n}|<m)$ for some $m\in \N$, then $(w_{n})_{n\in \N}$ contains at most $m$ different
elements, as $Y$ is an injection.  The pigeon hole principle now implies that (at least) one $w_{n_{0}}$ occurs infinitely often in $(w_{n})_{n\in \N}$, yielding an obviously convergent sub-sequence.  
In case $(\forall  m\in \N)(\exists n\in \N)( |w_{n}|\geq m)$, the sequence $(w_{n})_{n\in \N}$ enumerates the reals in $[0,1]$ (as $Y$ is a bijection), which is impossible by Theorem \ref{cant2}.
Thus, $(N, d)$ is a sequentially compact space; the function $G:N\di \R$ defined as $G(u)=|u|$ is continuous (in the same way as for $F$ above) but not bounded.  
This contradiction establishes that item \eqref{itemg} implies $\NBI_{[0,1]}$, and the same for items \eqref{itemgg}-\eqref{itemi}.
For item~\eqref{modulus}, the function $H(x,k):= \frac{1}{2^{|x|+k+2}}$ is a modulus of continuity for $G$.  

\smallskip

Seventh, for item \eqref{diniitem}, assume again $\neg\NBI_{[0,1]}$ and define $G_{n}(w)$ as $|w|$ in case $|w|\leq n$, and $0$ otherwise.  
As for $G$ above, $G_{n}$ is continuous and $\lim_{n\di \infty}G_{n}(w)=G(w)$ for $x\in N$. 
Since $G_{n}\leq G_{n+1}$ on $N$, item \eqref{diniitem} implies that the convergence is uniform, i.e.\ we have
\be\label{valse}\textstyle
(\forall k\in \N)(\exists m\in \N)(\forall w\in N)(\forall n\geq m)( |G_{n}(w)-G(w)|<\frac{1}{2^{k}}),
\ee
which is clearly false.  Indeed, take $k=1$ and let $m_{1}\in \N$ be as in \eqref{valse}.  Since $Y$ is surjective, $\IND_{1}$ provides $w_{1}\in N$ of length $m_{1}+1$, yielding $|G(w_{1})-G_{m_{1}}(w_{1})|=|(m_{1}+1)-0|>\frac{1}{2}$, contradicting \eqref{valse} and thus $\NBI_{[0,1]}$ follows from item \eqref{diniitem}.  For item \eqref{equi}, $(G_{n})_{n\in \N}$ is equicontinuous by the previous, but not uniformly equicontinuous, just like for item \eqref{diniitem} using a variation of \eqref{valse}.  For item \eqref{pinker}, the function $J(w):=\frac{1}{2^{|w|}} $ is continuous on $N$ in the same way as for $F, G$.  
However, assuming $\neg\NBI_{[0,1]}$, $J$ becomes arbitrarily small on $N$, contradicting item \eqref{pinker}. 
For item \eqref{AA}, define $J_{n}(w)$ as $J(w)$ if $|w|\leq n$, and $1$ otherwise.  Similar to the previous, $J_{n}$ converges to $J$, but not uniformly, i.e.\ item \eqref{AA} also implies $\NBI_{[0,1]}$. 
 
\smallskip

For item \eqref{itemlast2}, note that $O_{n}:= \{w\in N: |w|=n\} $ is open as $B_{d}^{M}(v, \frac{1}{2^{n+2}})\subset O_{n}$ in case $v\in O_{n}$.
Then $\cup_{n\in \N}O_{n}$ covers $N$, assuming $N$ (and $\neg\NBI_{[0,1]}$) as above.
However, there clearly is no finite sub-covering.  

\smallskip

Finally, for items \eqref{itemii}-\eqref{itemlast}, the above proof for items \eqref{iteme}-\eqref{itemf} goes through without modification. 
For item \eqref{itemiii}, note that $N$ is an infinite set in $(N, d)$ without limit point.  
The final sentence speaks for itself: one uses $(\exists^{3})$ and $(\mu^{2})$ to obtain a modulus of continuity.  
For $\eps=1$, the latter yields an uncountable covering, which has finite sub-covering assuming $(\exists^{3})$ by \cite{dagsamV}*{Theorem 4.1}.  
This immediately yields an upper bound while the supremum and maximum are obtained using the usual interval-halving technique using $(\exists^{3})$.
\end{proof}
We could restrict item \eqref{itemlast2} to \emph{R2-open} sets (\cites{dagsamVII, dagsamXV}), where the latter are open sets such that $x\in U$ implies $B(x, h_{U}(x))\subset U$ with the notation of Def.\ \ref{bchar}.

\subsection{Variations on a theme}\label{refi}
Lest the reader believe that third-order metric spaces are somehow irredeemable, we show that certain (very specific) variations 
of the items in Theorem \ref{klon} are provable in rather weak systems, sometimes assuming countable choice as in $\QFAC^{0,1}$ (Theorems \ref{typemon} and \ref{zlonk}).  We also show 
that certain items in Theorem \ref{klon} are just very hard to prove by deriving some of the new `Big' systems from \cite{dagsamXI, samBIG3, dagsamX, samBIG}, namely the Jordan decomposition theorem and the uncountability of $\R$ as in $\NIN_{[0,1]}$ (Theorem \ref{zionk}).

\smallskip

First, we establish the following theorem, which suggests a strong need for open sets as in Def.\ \ref{bchar} if we wish to prove basic properties of metric spaces in the base theory, potentially extended with the Big Five.
The fourth item should be contrasted with item \eqref{iteme} in Theorem \ref{klon}.  Many variations of the below results are of course possible based on the associated second-order results.   
\begin{thm}[$\RCAo$]\label{typemon}~
\begin{enumerate}
\renewcommand{\theenumi}{\alph{enumi}}
\item For strongly countably open $(M, d)$, a continuous $F:M\di \R$ is bounded.  
\item Dini's theorem for strongly countably-compact $(M, d)$.  
\item Pincherle's theorem for strongly countably-compact $(M, d)$.  
\item A metric space $(M, d)$ with the Cantor intersection property, is strongly countably-compact.  
\item The following are equivalent:
\begin{itemize}
\item[(e.1)]  weak K\"onig's lemma $\WKL_{0}$,
\item[(e.2)] for any weakly complete and \textbf{effectively} totally bounded metric space $(M, d)$ with $M\subset [0,1]$, a continuous $F:M\di \R$ is bounded above, 
\item[(e.3)] the previous item for sequentially continuous functions. 
\end{itemize}
\item The following are equivalent.
\begin{itemize}
\item[(f.1)]  arithmetical comprehension $\ACA_{0}$.
\item[(f.2)] any weakly complete and \textbf{effectively} totally bounded metric space $(M, d)$ with $M\subset [0,1]$, is sequentially compact. 
\end{itemize}
\end{enumerate}
\end{thm}
\begin{proof}
For the first item, since $F$ is continuous, the set $E_{n}:= \{ x\in M: |F(x)|>n \}$ is open and exists in the sense of Def.\ \ref{bchar}. Since $\cup_{n\in \N}E_{n}$ covers $(M, d)$, there is a finite sub-covering $\cup_{n\leq n_{0}}E_{n}$ for some $n_{0}\in \N$, implying $|F(x)|\leq n_{0}+1$ for all $x\in M$, i.e.\ $F$ is bounded as required.  

\smallskip

For the second item, let $F, F_{n}$ be as in Dini's theorem and define $G_{n}(w):=F(w)-F_{n}(w)$.  
Now fix $k\in \N$ and define $E_{n}:= \{w\in M:  G_{n}(w)<\frac{1}{2^{k}}   \}$.  The latter yields a countable open covering and 
one obtains uniform convergence from any finite sub-covering. 
For the third item, fix $F:M\di \R^{+}$ and define $E_{n}:=\{w\in M: F(w)>\frac{1}{2^{n}} \}  $. 
The proof proceeds as for the previous items. 

\smallskip

For the fourth item, this amounts to a manipulation of definitions.  For the fifth item, that (e.2) and (e.3) imply $\WKL_{0}$ is immediate by \cite{dagsamXIV}*{Theorem 2.8} for $M=[0,1]$ and \cite{kohlenbach2}*{Prop.\ 3.6}.   
For the downward implication, fix $F:M\di \R$ for $M\subset [0,1]$ as in item (e.2).  
In case $\neg(\exists^{2})$, all functions on $\R$ are continuous by \cite{kohlenbach2}*{Prop.~3.12}.  
By \cite{dagsamXIV}*{Theorem~2.8}, all (continuous) $[0,1]\di \R$-functions are bounded.  Since we may (also) view $F$ as a (continuous) function from reals to reals, $F$ is bounded on $[0,1]$ and hence $M$, i.e.\ this case is finished. 

\smallskip
 
In case $(\exists^{2})$, we follow the well-known proof to show that $(M, d)$ is sequentially compact.  Indeed, for a sequence $(x_{n})_{n\in \N}$ in $M$, define a sub-sequence as follows: $M$ can be covered by a finite number of balls with radius $1/2^{k}$ with $k=1$.  Find a ball with infinitely many elements of $(x_{n})_{n\in \N}$ inside (which can be done explicitly using $(\exists^{2})$) and choose $x_{n_{0}}$ in this ball to define $y_{0}:=x_{n_{0}}$. 
Now repeat the previous steps for $k>1$ and note that the resulting sequence is effectively Cauchy and hence convergent (by the assumptions on $M$).  
Hence, $(M, d)$ is sequentially compact and suppose $F:M\di \R$ is unbounded, i.e.\ $(\forall n\in \N)\underline{(\exists x\in M)}(F(x)>n)$.  
It is now important to note that the underlined quantifier can be replaced by a quantifier over $\N$ using the sequence $(w_{n})_{n\in \N}$ provided by $M$ being effectively totally bounded.  
Applying $\QFAC^{0, 0}$, included in $\RCAo$, there is a sequence $(x_{n})_{n\in \N}$ such that $|F(x_{n})|>n$.  
This sequence has a convergent sub-sequence, say with limit $y$, and $F$ is not continuous at $y$, a contradiction.  Thus, $F$ is bounded for both disjuncts of $(\exists^{2})\vee \neg(\exists^{2})$.  The equivalence involving $\ACA_{0}$ has a similar proof.  
\end{proof}
As emphasised in bold in the theorem, the final part of the proof seems to crucially depend on \emph{effective} totally boundedness.  
Indeed, by the first part of Theorem~\ref{klon}, item (e.3) of Theorem \ref{typemon} with `effectively' omitted, implies $\NBI_{[0,1]}$.
In other words, the equivalences in Theorem \ref{typemon} do not seem robust.  

\smallskip

Secondly, we show that certain items from Theorem \ref{klon} fit nicely with RM, assuming an extended base theory.
Other items turn out to be connected to the `new' Big systems studied in \cite{dagsamXI, samBIG, samBIG2}. 

\smallskip

We now show that certain items from Theorem \ref{klon} are provable assuming countable choice as in $\QFAC^{0,1}$.  
Thus, these items do not imply $\NIN_{[0,1]}$ as the latter is not provable in $\Z_{2}^{\omega}+\QFAC^{0,1}$.  
The third item should be contrasted with \cite{simpson2}*{III.2}.
Many results in RM do not go through in the absence of $\QFAC^{0,1}$, as studied at length in \cites{dagsamVII, dagsamV}.
\begin{thm}[$\RCAo+\QFAC^{0,1}$]\label{zlonk} 
The following are provable for $(M, d)$ any metric space with $M\subset \R$.   
\begin{itemize}
\item Items \eqref{itemg}, \eqref{itemgg}, \eqref{diniitem}, \eqref{equi}, \eqref{pinker}, \eqref{itemlast2}, \eqref{itemlast}, and \eqref{itemiii} from Theorem \ref{klon}.
\item The following are equivalent:
\begin{itemize}
\item  weak K\"onig's lemma $\WKL_{0}$,
\item the unit interval is strongly countably-compact.
\end{itemize}
\item The following are equivalent:
\begin{itemize}
\item  arithmetical comprehension $\ACA_{0}$,
\item a weakly complete and \textbf{effectively} totally bounded $(M, d)$ with $M\subset [0,1]$ is limit point compact.
\end{itemize}
\end{itemize}
\end{thm}
\begin{proof}
First of all, we prove item \eqref{itemg} from Theorem \ref{klon} in $\RCAo+\QFAC^{0,1}$. 
To this end, suppose the continuous function $F:M\di \R$ is unbounded, i.e.\ $(\forall n\in \N)(\exists w\in M)(|F(w)|>n)$.  
Applying $\QFAC^{0, 1}$, there is a sequence $(x_{n})_{n\in \N}$ such that $|F(w_{n})|>n$.  Since $(M, d)$ is assumed to be 
sequentially complete, let $(y_{n})_{n\in \N}$ be a convergent sub-sequence with limit $y\in M$.  Clearly, $F$ cannot be continuous at $y\in M$, a contradiction, which yields item \eqref{itemg}.   
Item \eqref{itemgg} is proved in the same way:  suppose $F$ is not uniformly continuous and apply $\QFAC^{0,1}$ to the latter statement to obtain a sequence.  
Then $F$ is not continuous at the limit of the convergent sub-sequence.  Items \eqref{diniitem}-\eqref{pinker} are proved in the same way.  
To prove item \eqref{itemlast2}, let $(O_{n})_{n\in \N}$ be a countable open covering of $M$ with $(\forall n\in \N)(\exists x\in M)( x\not \in \cup_{m\leq n}O_{m} )$.  
Apply $\QFAC^{0,1}$ to obtain a sequence $(x_{n})_{n\in \N}$, which has a convergent sub-sequence $(y_{n})_{n\in \N}$ by assumption, say with limit $y\in M$.  Then $y\in O_{n_{0}}$ for some $n_{0}\in \N$, which implies that $y_{n}$ is also eventually in $O_{n_{0}}$, a contradiction. 
To prove item~\eqref{itemlast}, let $(C_{n})_{n\in \N}$ be as in the sequential Cantor intersection property and apply $\QFAC^{0,1}$ to $(\forall n\in \N)(\exists x\in M)( x\in C_{n})$.  
The convergent sub-sequence has a limit $y\in \cap_{n\in \N} C_{n}$.  To prove item \eqref{itemiii}, let $X$ be an infinite set, i.e.\  $(\forall N\in \N)(\exists w^{1^{*}})(\forall i<|w|)(|w|=N\wedge w(i)\in X  )$. 
Now apply $\QFAC^{0,1}$ to obtain a sequence $(w_{n})_{n\in \N}$ in $X$.  Since $(M, d)$ is sequentially closed, the latter sequence has a convergent sub-sequence, the limit of which is a limit point of $X$.  

\smallskip

Secondly, the equivalence in the second item is proved in \cite{dagsamVII}*{Theorem 4.1}.  For the third item, the upwards implication is immediate for $M=[0,1]$.  
For the downwards implication, assume $(M, d)$ as in the final sub-item.  Theorem \ref{typemon} implies that $(M, d)$ is sequentially compact.  
As in the previous paragraph, an infinite set in $M$ now has a limit point.  
\end{proof}
A similar proof should go through for many of the other items in Theorem \ref{klon} and for $\QFAC^{0,1}$ replaced by $\NCC$ from \cite{dagsamIX}; the latter is provable in $\Z_{2}^{\Omega}$ while the former is not provable in $\ZF$. 

\smallskip

Secondly, the Jordan decomposition theorem is studied in \cite{dagsamXI, samBIG3} where various versions are shown to be equivalent to the enumeration principle for countable sets. 
Many equivalences exist for the following principle, elevating it to a new `Big' system, as shown in \cite{dagsamXI}.  
\begin{princ}[$\cocode_{0}$]
Let $A\subset [0,1]$ and $Y:[0,1]\di \N$ be such that $Y$ is injective on $A$. 
Then there is a sequence of reals $(x_{n})_{n\in \N}$ that includes $A$.   
\end{princ}
This principle is `explosive' in that $\ACAo+\cocode_{0}$ proves $\ATR_{0}$ and $\FIVE^{\omega}+\cocode_{0}$ proves $\SIX$ (see \cite{dagsamXI}*{\S4}).
As it turns out, the separability of metric spaces is similarly explosive.  
\begin{thm}[$\ACAo$]\label{zionk}~
\begin{itemize}
\item Item \eqref{itemf} or \eqref{itemii} from Theorem \ref{klon} implies $\cocode_{0}$.  
\item Item \eqref{itemf} or \eqref{itemii} for $M=[0,1]$ from Theorem \ref{klon} implies $\NIN_{[0,1]}$.  
\end{itemize}
\end{thm}
\begin{proof}
For the first item, let $Y:[0,1]\di \N$ be injective on $A\subset [0,1]$; without loss of generality, we may assume $0\in A$.  
Now define $d(x, y):= |\frac{1}{2^{Y(x)}}-\frac{1}{2^{Y(y)}}|$, $d(x,0)=d(0, x):=\frac{1}{2^{Y(x)}}$ for $x, y\ne 0$ and $d(0, 0):=0$. 
The metric space $(A, d)$ is countably-compact as $0\in B_{d}^{A}(x, r)$ implies $y\in B_{d}^{A}(x, r)$ for $y\in A$ with only finitely many exceptions (as $Y$ is injective on $A$).  
Similarly, $(A, d)$ is sequentially compact: in case a sequence $(z_{n})_{n\in \N}$ in $A$ has at most finitely many distinct elements, there is an obvious convergent/constant sub-sequence.  
Otherwise, $(z_{n})_{n\in \N}$ has a sub-sequence $(y_{n})_{n\in \N}$ such that $Y(y_{n})$ becomes arbitrary large with $n$ increasing; this sub-sequence is readily seen to converge to $0$.   

\smallskip

Now let $(x_{n})_{n\in\N}$ be the sequence provided by item \eqref{itemf} or \eqref{itemii} of Theorem~\ref{klon}, implying $(\forall x\in A)(\exists n\in \N)( d(x, x_{n})<\frac{1}{2^{Y(x)+1}})$ by taking $k=Y(x)+1$.  
The latter formula implies 
\be\label{jilogz}\textstyle
(\forall x\in A)(\exists n\in \N)(x\ne_{\R} 0\di  |\frac{1}{2^{Y(x)}}-\frac{1}{2^{Y(x_{n})}}|<_{\R}\frac{1}{2^{Y(x)+1}}) 
\ee
by definition.  Note that $x_{n}$ from \eqref{jilogz} cannot be $0$ by the definition of the metric $d$.  
Clearly, $|\frac{1}{2^{Y(x)}}-\frac{1}{2^{Y(x_{n})}}|<\frac{1}{2^{Y(x)+1}}$ is only possible if $Y(x)=Y(x_{n})$, implying $x=_{\R}x_{n}$. 
Hence, we have shown that $(x_{n})_{n\in \N}$ lists all reals in $A\setminus\{0\}$.
The same proof now yields the second item for $A=[0,1]$ as Theorem \ref{cant2} implies the reals cannot be enumerated. 
\end{proof}
In conclusion, the coding of metric spaces does distort the logical properties of basic properties of continuous functions on metric spaces by Theorem \ref{klon}.  
This is established by deriving $\NBI_{[0,1]}$ while noting that $\NIN_{[0,1]}$ generally does not follow by Theorem \ref{zlonk}.  
The latter also shows that in an enriched base theory, one can obtain `rather vanilla' RM.  
By contrast, other properties of metric spaces imply new `Big' systems, as is clear from Theorem \ref{zionk}.

\section{Foundational musings}\label{fouind}

\subsection{Thoughts on coding}
The results in this paper have implications for the coding of higher-order objects in second-order RM, as discussed in this section. 

\smallskip

First of all, our results shed new light on the following problem from \cite{fried5}*{p.\ 135}.
\begin{quote}
PROBLEM. [\dots] Show that 
Simpson's neighborhood condition coding of partial continuous functions
between complete separable metric spaces is ``optimal". 
\end{quote}
A coding is called \emph{optimal} in \cite{fried5} in case $\RCA_{0}$ can prove `as much as possible', i.e.\ as many as possible of the basic properties of the coding can be established in $\RCA_{0}$.  
Theorem \ref{klon} show that without separability, basic properties of continuous functions on compact metric spaces are no longer provable from second-order (comprehension) axioms.  
Thus, separability is an essential ingredient \emph{if} one wishes to study these matters using second-order arithmetic/axioms.  

\smallskip

Secondly, second-order (comprehension) axioms can establish many (third-order) theorems about continuous \emph{and} discontinuous functions
on the reals (see \cite{dagsamXIV, samBIG3}), assuming $\RCAo$.  
Hence, large parts of (third-order) real analysis can be developed using second-order comprehension axioms in a weak third-order background theory, namely $\RCAo$, using little-to-no-coding.  
The same does not hold for continuous functions on compact metric spaces by the above results.  In particular, Theorem \ref{typemon} suggests we have to choose a \emph{very specific} representation, namely `weakly complete and effectively totally bounded' to obtain third-order statements that are classified in the Big Five.  Indeed, Theorem \ref{klon} implies that many (most?) other variations are not provable from second-order (comprehension) axioms. 

\smallskip

In conclusion, our results show that separability is an essential ingredient \emph{if} one wishes to study these matters using second-order arithmetic/axioms.  
However, our results also show that this is a \emph{very specific} choice that is `non-standard' in the sense that many variations cannot be established using second-order arithmetic/axioms.

\subsection{Set theory and ordinary mathematics}
In this section, we explore a theme introduced in \cite{samBIG2}. 
Intuitively speaking, we collect evidence for a parallel between our results and some central results in set theory.
 Formulated slightly differently, one could say that interesting phenomena in set theory have `miniature versions' to be found in third-order arithmetic, or that the seeds for interesting phenomena in set theory can already be found in third-order arithmetic.

\smallskip

First of all, the cardinality of $\R$ is mercurial in nature: the famous work of G\"odel (\cite{goeset}) and Cohen (\cite{cohen1, cohen2}) shows that the \emph{Continuum Hypothesis} cannot be proved or disproved in $\ZFC$, i.e.\ Zermelo-Fraenkel set theory with $\AC$, the usual foundations of mathematics.  In particular, the exact cardinality of $\R$ cannot be established in $\ZFC$.  
A parallel observation in higher-order RM is that $\Z_{2}^{\omega}+\QFAC^{0,1}$ cannot prove that $\R$ is uncountable in the sense of there being no no injection from $\R$ to $\N$ (see \cite{dagsamX} for details).  
In a conclusion, the cardinality of $\R$ has a particularly mercurial nature, in both set theory and higher-order arithmetic.

\smallskip

Secondly, many standard results in mainstream mathematics are not provable in $\ZF$, i.e.\ $\ZFC$ with $\AC$ removed, as explored in great detail \cite{heerlijkheid}.
The absence of $\AC$ is even said to lead to \emph{disasters} in topology and analysis (see \cite{kermend}).  A parallel phenomenon was observed in \cites{dagsamVII, dagsamV}, namely 
that certain rather basic equivalences go through over $\RCAo+\QFAC^{0,1}$, but not over $\Z_{2}^{\omega}$.  

\smallskip

Examples include the equivalence between compactness results and local-global principles, which are intimately related according to Tao (\cite{taokejes}).  
In this light, it is fair to say that disasters happen in both set theory and higher-order arithmetic in the absence of $\AC$.
It should be noted that $\QFAC^{0,1}$ (not provable in $\ZF$) can be replaced by $\NCC$ from \cite{dagsamIX} (provable in $\Z_{2}^{\Omega}$) in the aforementioned. 

\smallskip

Thirdly, we discuss the essential role of $\AC$ in measure and integration theory, which leads to rather concrete parallel observations in higher-order arithmetic.
Indeed, the full pigeonhole principle for measure spaces is not provable in $\ZF$, which immediately follows from e.g.\ \cite{heerlijkheid}*{Diagram~3.4}.
A parallel phenomenon in higher-order arithmetic (see \cite{samBIG2}) is that even the restriction to closed sets, namely $\PHP_{[0,1]}$ cannot be proved in $\Z_{2}^{\omega}+\QFAC^{0,1}$ (but $\Z_{2}^{\Omega}$ suffices). 

\smallskip

A more `down to earth' observation pertains to the intuitions underlying the Riemann and Lesbesgue integral.
Intuitively, the integral of a non-negative function represents the area under the graph; thus, if the integral is zero, then this function must be zero for `most' reals.  
Now, $\AC$ is needed to establish this intuition for the Lebesgue integral (\cite{hahakaka}).
Similarly, \cite{samBIG2}*{Theorem 3.8} establishes the parallel observation that this intuition for the \emph{Riemann} integral cannot be proved in $\Z_{2}^{\omega}+\QFAC^{0,1}$ (but $\Z_{2}^{\Omega}$ suffices as usual).  

\smallskip

Fourth, the \emph{pointwise} equivalence between sequential and `epsilon-delta' continuity cannot be proved in $\ZF$ while $\RCAo+\QFAC^{0,1}$ suffices for functions on Baire space (see \cite{kohlenbach2}).  
A parallel observation is provided by (the proof of) Theorem \ref{klon}, namely that the following statement is not provable in $\Z_{2}^{\omega}$:
\begin{center}
\emph{for countably-compact $(M, d)$ and sequentially continuous $F:M\di \R$, $F$ is continuous on $M$.}
\end{center}
Thus, the \emph{global} equivalence between sequential and `epsilon-delta' continuity on metric spaces cannot be proved in $\Z_{2}^{\omega}$.  
In other words, the exact relation between sequential and `epsilon-delta' continuity is hard to pin down, both in set theory and third-order arithmetic. 

\begin{ack}\rm 
My research was supported by the \emph{Klaus Tschira Boost Fund} via the grant Projekt KT43.
The initial ideas for this paper were developed in my 2022 Habilitation thesis at TU Darmstadt (\cite{samhabil}) under the guidance of Ulrich Kohlenbach.  
We express our gratitude towards these persons and institutions.   
\end{ack}


\begin{bibdiv}
\begin{biblist}
\bib{diniberg2}{article}{
  author={Berger, Josef},
  author={Schuster, Peter},
  title={Classifying Dini's theorem},
  journal={Notre Dame J. Formal Logic},
  volume={47},
  date={2006},
  number={2},
  pages={253--262},
}

\bib{bish1}{book}{
  author={Bishop, Errett},
  title={Foundations of constructive analysis},
  publisher={McGraw-Hill},
  date={1967},
  pages={xiii+370},
}

\bib{brownphd}{book}{
  author={Brown, Douglas K.},
  title={Functional analysis in weak subsystems of second-order arithmetic},
  year={1987},
  publisher={PhD Thesis, The Pennsylvania State University, ProQuest LLC},
}

\bib{browner2}{article}{
  author={Brown, Douglas K.},
  title={Notions of closed subsets of a complete separable metric space in weak subsystems of second-order arithmetic},
  conference={ title={Logic and computation}, address={Pittsburgh, PA}, date={1987}, },
  book={ series={Contemp. Math.}, volume={106}, publisher={Amer. Math. Soc., Providence, RI}, },
  date={1990},
  pages={39--50},
}

\bib{browner}{article}{
  author={Brown, Douglas K.},
  title={Notions of compactness in weak subsystems of second order arithmetic},
  conference={ title={Reverse mathematics 2001}, },
  book={ series={Lect. Notes Log.}, volume={21}, publisher={Assoc. Symbol. Logic}, },
  date={2005},
  pages={47--66},
}

\bib{cohen1}{article}{
  author={Cohen, Paul},
  title={The independence of the continuum hypothesis},
  journal={Proc. Nat. Acad. Sci. U.S.A.},
  volume={50},
  date={1963},
  pages={1143--1148},
}

\bib{cohen2}{article}{
  author={Cohen, Paul},
  title={The independence of the continuum hypothesis. II},
  journal={Proc. Nat. Acad. Sci. U.S.A.},
  volume={51},
  date={1964},
  pages={105--110},
}

\bib{dinipi}{book}{
  author={U. {Dini}},
  title={{Fondamenti per la teorica delle funzioni di variabili reali}},
  year={1878},
  publisher={{Nistri, Pisa}},
}

\bib{dinipi2}{book}{
  author={U. {Dini}},
  title={{Grundlagen f\"ur eine Theorie der Functionen einer ver\"anderlichen reelen Gr\"osse}},
  year={1892},
  publisher={{Leipzig, B.G. Teubner}},
}

\bib{damurm}{book}{
  author={Dzhafarov, Damir D.},
  author={Mummert, Carl},
  title={Reverse Mathematics: Problems, Reductions, and Proofs},
  publisher={Springer Cham},
  date={2022},
  pages={xix, 488},
}

\bib{fried5}{article}{
  author={Friedman, Harvey},
  author={Simpson, Stephen G.},
  title={Issues and problems in reverse mathematics},
  conference={ title={Computability theory and its applications}, address={Boulder, CO}, date={1999}, },
  book={ series={Contemp. Math.}, volume={257}, publisher={Amer. Math. Soc.}, place={Providence, RI}, },
  date={2000},
  pages={127--144},
}

\bib{goeset}{article}{
  author={G\"odel, Kurt},
  title={The Consistency of the Axiom of Choice and of the Generalized Continuum-Hypothesis},
  journal={Proceedings of the National Academy of Science},
  year={1938},
  volume={24},
  number={12},
  pages={556-557},
}

\bib{grayk}{article}{
  author={Gray, Robert},
  title={Georg Cantor and transcendental numbers},
  journal={Amer. Math. Monthly},
  volume={101},
  date={1994},
  number={9},
  pages={819--832},
}

\bib{heerlijkheid}{book}{
  author={Herrlich, Horst},
  title={Axiom of choice},
  series={Lecture Notes in Mathematics},
  volume={1876},
  publisher={Springer},
  date={2006},
  pages={xiv+194},
}

\bib{hunterphd}{book}{
  author={Hunter, James},
  title={Higher-order reverse topology},
  note={Thesis (Ph.D.)--The University of Wisconsin - Madison},
  publisher={ProQuest LLC, Ann Arbor, MI},
  date={2008},
  pages={97},
}

\bib{hahakaka}{article}{
  author={Kanovei, Vladimir},
  author={Katz, Mikhail},
  title={A positive function with vanishing Lebesgue integral in Zermelo-Fraenkel set theory},
  journal={Real Anal. Exchange},
  volume={42},
  date={2017},
  number={2},
  pages={385--390},
}

\bib{kermend}{article}{
  author={Keremedis, Kyriakos},
  title={Disasters in topology without the axiom of choice},
  journal={Arch. Math. Logic},
  volume={40},
  date={2001},
  number={8},
}

\bib{kohlenbach4}{article}{
  author={Kohlenbach, Ulrich},
  title={Foundational and mathematical uses of higher types},
  conference={ title={Reflections on the foundations of mathematics}, },
  book={ series={Lect. Notes Log.}, volume={15}, publisher={ASL}, },
  date={2002},
  pages={92--116},
}

\bib{kohlenbach2}{article}{
  author={Kohlenbach, Ulrich},
  title={Higher order reverse mathematics},
  conference={ title={Reverse mathematics 2001}, },
  book={ series={Lect. Notes Log.}, volume={21}, publisher={ASL}, },
  date={2005},
  pages={281--295},
}

\bib{kohlenbach3}{book}{
  author={Kohlenbach, Ulrich},
  title={Applied proof theory: proof interpretations and their use in mathematics},
  series={Springer Monographs in Mathematics},
  publisher={Springer-Verlag},
  place={Berlin},
  date={2008},
  pages={xx+532},
}

\bib{kopo}{article}{
  author={Kohlenbach, Ulrich},
  title={On the logical analysis of proofs based on nonseparable Hilbert space theory},
  conference={ title={Proofs, categories and computations}, },
  book={ series={Tributes}, volume={13}, publisher={Coll. Publ., London}, },
  isbn={978-1-84890-012-7},
  date={2010},
  pages={131--143},
}

\bib{longmann}{book}{
  author={Longley, John},
  author={Normann, Dag},
  title={Higher-order Computability},
  year={2015},
  publisher={Springer},
  series={Theory and Applications of Computability},
}

\bib{neeman}{article}{
  author={Neeman, Itay},
  title={Necessary use of $\Sigma ^1_1$ induction in a reversal},
  journal={J. Symbolic Logic},
  volume={76},
  date={2011},
  number={2},
  pages={561--574},
}

\bib{dagsam}{article}{
  author={Normann, Dag},
  author={Sanders, Sam},
  title={Nonstandard Analysis, Computability Theory, and their connections},
  journal={Journal of Symbolic Logic},
  volume={84},
  number={4},
  pages={1422--1465},
  date={2019},
}

\bib{dagsamII}{article}{
  author={Normann, Dag},
  author={Sanders, Sam},
  title={The strength of compactness in Computability Theory and Nonstandard Analysis},
  journal={Annals of Pure and Applied Logic, Article 102710},
  volume={170},
  number={11},
  date={2019},
}

\bib{dagsamIII}{article}{
  author={Normann, Dag},
  author={Sanders, Sam},
  title={On the mathematical and foundational significance of the uncountable},
  journal={Journal of Mathematical Logic, \url {https://doi.org/10.1142/S0219061319500016}},
  date={2019},
  volume = {19},
number = {01},
pages = {1950001},

}

\bib{dagsamVI}{article}{
  author={Normann, Dag},
  author={Sanders, Sam},
  title={The Vitali covering theorem in Reverse Mathematics and computability theory},
  journal={Submitted, arXiv: \url {https://arxiv.org/abs/1902.02756}},
  date={2019},
}

\bib{dagsamVII}{article}{
  author={Normann, Dag},
  author={Sanders, Sam},
  title={Open sets in Reverse Mathematics and Computability Theory},
  journal={Journal of Logic and Computation},
  volume={30},
  number={8},
  date={2020},
  pages={pp.\ 40},
}

\bib{dagsamV}{article}{
  author={Normann, Dag},
  author={Sanders, Sam},
  title={Pincherle's theorem in reverse mathematics and computability theory},
  journal={Ann. Pure Appl. Logic},
  volume={171},
  date={2020},
  number={5},
  pages={102788, 41},
}

\bib{dagsamIX}{article}{
  author={Normann, Dag},
  author={Sanders, Sam},
  title={The Axiom of Choice in Computability Theory and Reverse Mathematics},
  journal={Journal of Logic and Computation},
  volume={31},
  date={2021},
  number={1},
  pages={297-325},
}

\bib{dagsamXI}{article}{
  author={Normann, Dag},
  author={Sanders, Sam},
  title={On robust theorems due to Bolzano, Jordan, Weierstrass, and Cantor in Reverse Mathematics},
  journal={Journal of Symbolic Logic, DOI: \url {doi.org/10.1017/jsl.2022.71}},
  pages={pp.\ 51},
  date={2022},
}

\bib{dagsamX}{article}{
  author={Normann, Dag},
  author={Sanders, Sam},
  title={On the uncountability of $\mathbb {R}$},
  journal={Journal of Symbolic Logic, DOI: \url {doi.org/10.1017/jsl.2022.27}},
 VOLUME = {87},
      YEAR = {2022},
    NUMBER = {4},
     PAGES = {1474--1521},
}

\bib{dagsamXIV}{article}{
  author={Normann, Dag},
  author={Sanders, Sam},
  title={The Biggest Five of Reverse Mathematics},
  journal={Journal for Mathematical Logic, doi: \url {https://doi.org/10.1142/S0219061324500077}},
  pages={pp.\ 56},
  date={2023},
}

\bib{dagsamXV}{article}{
  author={Normann, Dag},
  author={Sanders, Sam},
  title={On the computational properties of open sets},
  journal={Submitted, arxiv: \url {https://arxiv.org/abs/2401.09053}},
  pages={pp.\ 26},
  date={2024},
}

\bib{tepelpinch}{article}{
  author={Pincherle, Salvatore},
  title={Sopra alcuni sviluppi in serie per funzioni analitiche (1882)},
  journal={Opere Scelte, I, Roma},
  date={1954},
  pages={64--91},
}

\bib{yamayamaharehare}{article}{
  author={Sakamoto, Nobuyuki},
  author={Yamazaki, Takeshi},
  title={Uniform versions of some axioms of second order arithmetic},
  journal={MLQ Math. Log. Q.},
  volume={50},
  date={2004},
  number={6},
  pages={587--593},
}

\bib{samhabil}{book}{
  author={Sam, Sanders},
  title={Some contributions to higher-order Reverse Mathematics},
  year={2022},
  publisher={Habilitationsschrift, TU Darmstadt},
}

\bib{samrep}{article}{
   author={Sanders, Sam},
   title={Representations and the foundations of mathematics},
   journal={Notre Dame J. Form. Log.},
   volume={63},
   date={2022},
   number={1},
   pages={1--28},
}

\bib{samhabil}{book}{ 
author={Sanders, Sam},
title={Some contributions to higher-order Reverse Mathematics},
year={2022},
publisher={Habilitationsschrift, TU Darmstadt}
}


\bib{samBIG}{article}{
  author={Sanders, Sam},
  title={Big in Reverse Mathematics: the uncountability of the real numbers},
  year={2023},
  journal={Journal of Symbolic Logic, doi:\url {https://doi.org/10.1017/jsl.2023.42}},
  pages={pp.\ 26},
}

\bib{samBIG2}{article}{
  author={Sanders, Sam},
  title={Big in Reverse Mathematics: measure and category},
  year={2023},
  journal={Journal of Symbolic Logic, doi: \url {https://doi.org/10.1017/jsl.2023.65}},
  pages={pp.\ 44},
}

\bib{samBIG3}{article}{
  author={Sanders, Sam},
  title={Approximation theorems and Reverse Mathematics},
  year={2023},
  journal={To appear in Journal of Symbolic Logic, arxiv: \url {https://arxiv.org/abs/2311.11036}},
  pages={pp.\ 25},
}

\bib{samBIG4}{article}{
  author={Sanders, Sam},
  title={Riemann integration and Reverse Mathematics},
  year={2023},
  journal={Submitted},
  pages={pp.\ 25},
}

\bib{simpson2}{book}{
  author={Simpson, Stephen G.},
  title={Subsystems of second order arithmetic},
  series={Perspectives in Logic},
  edition={2},
  publisher={CUP},
  date={2009},
  pages={xvi+444},
}

\bib{sigohi}{incollection}{
    Author = {Simpson, Stephen G.},
    Title = {{The G\"odel hierarchy and reverse mathematics.}},
    BookTitle = {{Kurt G\"odel. Essays for his centennial}},
    Pages = {109--127},
    Year = {2010},
    Publisher = {Cambridge University Press},
}

\bib{stillebron}{book}{
  author={Stillwell, J.},
  title={Reverse mathematics, proofs from the inside out},
  pages={xiii + 182},
  year={2018},
  publisher={Princeton Univ.\ Press},
}

\bib{taokejes}{collection}{
  author={Tao, Terence},
  title={{Compactness and Compactification}},
  editor={Gowers, Timothy},
  pages={167--169},
  year={2008},
  publisher={The Princeton Companion to Mathematics, Princeton University Press},
}


\end{biblist}
\end{bibdiv}
\bye